\documentclass[a4paper,12pt]{article}
\usepackage{amsmath,amsfonts,amsthm,amssymb,verbatim,graphicx,color}
\usepackage{enumerate}
\usepackage{comment}

\title{Formulas in connection with parameters related to convexity of paths on three vertices: caterpillars and unit interval graphs}

\author{Luc\'ia M. Gonz\'alez\footnote{Instituto de Ciencias, Universidad Nacional de General Sarmiento, Los Polvorines, Buenos Aires, Argentina} \and Luciano N. Grippo\footnote{Instituto de Ciencias, Universidad Nacional de General Sarmiento, Los Polvorines, Buenos Aires, Argentina and Consejo Nacional de Investigaciones Cient\'ificas y T\'ecnicas (CONICET), Argentina}  \and Mart\'in D. Safe\footnote{Departamento de Matem\'atica, Universidad Nacional del Sur (UNS), Bah\'ia Blanca, Argentina and INMABB, Universidad Nacional del Sur (UNS)-CONICET, Bah\'ia Blanca, Argentina}}

\newtheorem{theorem}{Theorem}
\newtheorem{coro}{Corollary}

\newtheorem{lem}{Lemma}
\newtheorem{remark}{Remark}
\theoremstyle{definition}

\newtheorem{prop}{Proposition}

\begin{document}
\maketitle
\begin{verbatim}\end{verbatim}\vspace{2.5cm}

\begin{abstract}
We present formulas to compute the $P_3$-geodetic number, the $P_3$-hull number and the percolation time for a caterpillar, in terms of certain sequences associated with it. In addition, we find a connection between the percolation time of a unit interval graph and a parameter involving the diameter of a unit interval graph related to it. Finally, we present a hereditary graph class, defined by  forbidden induced subgraphs, such that its percolation time is equal to one. 

\end{abstract}


\section{Introduction}\label{intro}

The convexity generated by paths of length two has been widely studied in the specialized literature. From an algorithmic perspective, there are several results in connection with different parameters related to convexity in graphs, in particular vinculated to $P_3$-convexity. Centeno et al. proved that is NP-complete deciding whether the $P_3$-hull number of a graph is at most $k$ \cite{Centeno2011}. Another parameter in connection with a convexity is the geodetic number of a graph that, under the $P_3$-convexity, agrees with the $2$-domination number~\cite{centeno}. An interesting problem, considered for many researches, is determining the percolation time of a graph. Interesting enough is the problem of deciding whether the percolation time is at least $k$. This decision problem can be solved in polinomial time if $k=4$. Nevertheless, the problem becomes NP-complete if $k=5$ \cite{Marcilon2018}. In this article we study the existence of formulas to compute these parameters. Caterpillars, which are precisely those acyclic connected interval graphs, are considered by giving a formula for each of the previously mentioned parameters. In the case of a unit interval graph, we find a formula for the percolation time in terms of the diameter of certain unit interval graphs related to it.

All graphs considered in this article are finite, undirected, without loops, and without multiple edges. All graph-theoretic concepts and definitions not given here can be found in~\cite{West01}. Let $G$ be a graph. Let us denote by $V(G)$ and $E(G)$ its vertex set and edge set, respectively. Given $u$ and $v$ in $V(G)$, we say that $u$ is \emph{adjacent} to $v$ if $uv\in E(G)$. The \emph{neighborhood} of a vertex $u$, denoted $N_{G}(u)$, is the set $\{v\in V(G)\colon\,uv\in E(G)\}$, and $N_G[v]$ stands for the set $N_G(v)\cup\{v\}$. If $X$ is a finite set, $\vert X\vert$ denotes its cardinality. The \emph{degree} of a vertex $u$ is the cardinality of its neighborhood (i.e., $|N_G(u)|$) and it is denoted by $d_{G}(u)$. The \emph{length} of a path is its number of edges. The \emph{distance} between two vertices $u$ and $v$ of $G$, denoted $d_G(u,v)$, is the minimum length among all paths having $u$ and $v$ as their endpoints. The \emph{diameter} of $G$, when it is connected, denoted $\mathrm{diam}(G)$, is the maximum $d(u,v)$ among all pairs $u,v\in V(G)$, a path $P$ having $\mathrm{diam}(G)$ edges es called a \emph{diameter path}. A \emph{pendant vertex} is a vertex of degree $1$. A \emph{support vertex} is a vertex adjacent to a pendant vertex.  A \emph{cut vertex} of $G$ is a vertex $v$ such that the number of connected component of $G-v$ is greater than the number of connected components of $G$. An \emph{independet set} of $G$ is a set of pairwise nonadjacent vertices. The maximum cardinality of an independent set of $G$ is denoted by $\alpha(G)$.  A \emph{tree} is a connected acyclic graph. A  \emph{leaf} of a tree is a pendant vertex of it. A \emph{caterpillar} is a tree such that the removal of all its pendant vertices turns out to be a path.

A \emph{convexity} of a graph $G$ is a pair $(V(G),\mathcal C)$ where $\mathcal C$ is a family of subsets of $V(G)$ satisfying all the following conditions: $ \emptyset \in \mathcal C$, $V(G)\in \mathcal C$, and $\mathcal C$ is closed under intersections; i.e., $V_1\cap V_2\in \mathcal C$ for each $V_1,V_2\in \mathcal C$. Each set of the family $\mathcal C$ is called a \emph{$\mathcal C$-convex set}. Let $\mathcal P$ be a set of paths in $G$ and let $S\subseteq V(G)$. If $u$ and $v$ are two vertices of $G$, then the \emph{$\mathcal P$-interval of $u$ and $v$}, denoted $I_{\mathcal P}[u,v]$, is the set of all vertices lying in some path $P\in \mathcal P$ having $u$ and $v$ as endpoints. Let $I_{\mathcal P}[S]=\bigcup_{u,v\in S}{I_{\mathcal P}[u,v]}$. Let $\mathcal C$ be the family of all sets $S$ of vertices of $G$ such that, for each path $P\in \mathcal P$ whose endpoints belong to $S$, every vertex of $P$ also belongs to $S$; i.e., $\mathcal C$ consists of those subsets $S$ of $V(G)$ such that $I_{\mathcal P}[S]=S$. It is easy to show that $(V(G),\mathcal C)$ is a convexity of $G$ and $\mathcal C$ is called the \emph{path convexity generated by $\mathcal P$}. The \emph{$P_3$-convexity} is the path convexity generated by the set of all paths of length two. Equivalently, a \emph{$P_3$-convex set} is a set $S\subseteq V(G)$ such that for each vertex $v\in V(G)-S$, $v$ has at most one neighbor in $S$. The \emph{$P_3$-hull set} of a set $R\subseteq V(G)$ is the minimum $P_3$-convex set of $G$ containing $R$. A \emph{$P_3$-hull set} of $G$ is a set of vertices whose $P_3$-hull set is $V(G)$, and the minimum cardinality of a $P_3$-hull set of $G$,  denoted $h(G)$, is the \emph{$P_3$-hull number}. If $R$ is a $P_3$-hull set of $G$, we also say that \emph{$R$ percolates $G$}. In order to decide whether a set $R$ percolates $G$, we may build a sequence $R_0, R_1, R_2,\ldots$ in which $R_0=R$ and $R_{i+1}$ is obtained from $R_i$ by adding those vertices of $G$ having at most one neighbor in $\bigcup_{j=1}^{i-1} R_j$ and at least two neighbors in $\bigcup_{j=1}^i R_j$. If there exists some $t$ such that $R_t=V(G)$, then $R$ percolates $G$ and we define $\tau_R(G)$ as the minimum $t$ such that $R_t=V(G)$. We also define the \emph{percolation time of $G$} as $\tau(G)=\max\{\tau_R(G)\colon\,R$ percolates $G$\}. Let $G$ be a graph and let $v\in V(G)$. We say that \emph{$S$ percolates $v$} if $v\in I^k_{P_3}[S]$ for some integer $k$, and $\tau_S(v)$  stands for the minimum $k$ with this property. If $A\subseteq V(G)$ and $S$ percolates every vertex $v\in A$, $\tau_S(A)$ stands for the maximum $t_S(w)$ among all vertices $w\in A$. When the context is clear enough, we will use $\tau(v)$ and $\tau(A)$ for short.  A set $S\subseteq V(G)$ is called \emph{$P_3$-geodetic} if for each vertex $v\in V(G)$, $v\in S$ or $v$ has at least two neighbors in $S$. The minimum cardinality of a $P_3$-geodetic set of $G$, denoted $g(G)$, is called the \emph{$P_3$-geodetic number} of $G$. Notice that, by definition, a $P_3$-geodetic set is also a $P_3$-hull set and thus $h(G)\leq g(G)$ for every graph $G$. In this article we will deal with $P_3$-convexity only. So, from now on, in some cases, we will omit the `$P_3$-' prefix and $\mathcal P$ will denote the set of paths on three vertices.

A  \emph{unit interval graph} is a graph such that there exists one-to-one assignment between its vertex set and a family of closed intervals of unit length in the real line, such that two different vertices are adjacent if and only if their corresponding intervals intersect. Equivalently, $G$ is a unit interval graph if there exists a linear order, called \emph{unit interval order}, of its vertices so that the closed neighborhood of each vertex is an interval  under that order; i.e, a sequence of consecutive vertices.
 
This article is organized as follows. In Section~\ref{formulas} we deal with the geodetic number and the hull number of  a caterpillar. More precisely, in Subsection~\ref{sec:caterpillar}, we associate a sequence of positive integers to every caterpillar, called basic sequence, used along the rest of the section: in Subsections~\ref{sec:geodetic_number} and~\ref{sec:hull_number} we present a closed formula for the $P_3$-geodetic number and the $P_3$-hull number of a caterpillar, respectively, in terms of its associated basic sequence. Section~\ref{sec: percolation} is devoted to give formulas for the percolation time for caterpillars and unit interval graphs. In Subsection~\ref{subsec:PTCG} we study the percolation time of a caterpillar and in Subsection~\ref{subsec: UIG} we summarize some known results on unit interval graphs and then we deal with the problem of finding the percolation time of a unit interval graph. Finally, in Section~\ref{sec:special graph} we present a hereditary graph class for which the percolation time of any graph in the class is 1.

\section{Computing the geodetic number and hull number of a caterpillar}\label{formulas}

\subsection{Basic sequence of a caterpillar}\label{sec:caterpillar}

A \emph{dominating set} of a graph $G$ is a set $S$ of vertices of $G$ such that every vertex $v\in V(G)\setminus S$ has at least one neighbor on $S$. A \emph{dominating path} in a graph $G$ is a path $P$ whose vertex set is a dominating set of $G$. Let $T$ be a caterpillar. By definition, $T$ has some dominating path $P$. Let $P=v_1,\ldots,v_k$ be chosen so that its endpoints are leaves of $T$. For each vertex $w$ of $T$, we define the \emph{reduced degree of $w$ in $T$} by $\tilde d_T(w)=\min\{d_T(w),4\}$. The sequence $\tilde d_T(v_1),\tilde d_T(v_2),\ldots,\tilde d_T(v_k)$, denoted by $s(T)$, is called a \emph{reduced degree sequence of $T$}. Since a dominating path of $T$ having pendant endpoints is unique up to the choice of its endpoints, the reduced degree sequence of $T$ is well-defined. Notice also that this sequence is unique up to reversing the order of the vertices of $P$.  By construction, the first and the last term of $s(T)$ are equal to $1$ (i.e., $\tilde d_T(v_1)=\tilde d_T(v_k)=1$). 

The \emph{length} of $\alpha$, denoted $\vert\alpha\vert$, is the number of terms of $\alpha$. For each $i\in\{1,\ldots,\vert\alpha\vert\}$, we denote by $\alpha_i$ the $i$-th term of $\alpha$. Let $\alpha^1$ and $\alpha^2$ be two sequences of lengths $n_1$ and $n_2$, respectively. The \emph{concatenation of $\alpha^1$ and $\alpha^2$}, denoted by $\alpha^1\alpha^2$, is the sequence $\beta$ of length $n_1+n_2$ characterized by $\beta_i=\alpha_i^1$ for each $i\in\{1,2,\ldots,n_1\}$ and $\beta_{n_1+i}=\alpha^2_i$ for each $i\in\{1,2,\ldots,n_2\}$. If either $\alpha^1$ or $\alpha^2$ is the empty sequence, then $\alpha^1\alpha^2$ is either $\alpha^2$ or $\alpha^1$, respectively. Let $\Gamma$ be the family of (possibly empty) sequences $\gamma$ such that all the terms of $\gamma$ are equal to $4$. A \emph{basic sequence} $\lambda$ is a sequence such that either $\lambda\in \{1, 21, 22, 23, 24 \}$ (where each element of the set should be thought as a sequence), or $\lambda$ is one of the sequences $x\gamma1$, $x\gamma2$, $x\gamma{31}$, $x\gamma{32}$, $x\gamma{33}$, or $x\gamma{34}$ for some $x\in \{3,4\}$ and some $\gamma\in\Gamma$. We denote by $\Lambda$ the family of all basic sequences. A \emph{prefix} of a sequence $\alpha$ is a sequence $\beta$ such that $\alpha=\beta\eta$ for some (possibly empty) sequence $\eta$. If $\eta$ is nonempty, then $\beta$ is a \emph{proper prefix of $\alpha$}. Notice that no basic sequence is a proper prefix of another basic sequence.

\begin{lem}\label{lem:1} Let $\sigma$ be a finite sequence whose terms belong to $\{1,2,3,4\}$ and whose last term is $1$. Then, there exists some integer $p\geq 0$ and some $\lambda^0,\lambda^1,\ldots,\lambda^p\in\Lambda$ such that $\sigma=\lambda^0\lambda^1\cdots\lambda^p$. Moreover, the integer $p$ and sequences $\lambda^0,\lambda^1,\ldots,\lambda^p$ are uniquely determined.\end{lem}
\begin{proof}
We prove the lemma by induction on the length of $\sigma$. Since the last term of $\sigma$ is equal to $1$, it is possible to define a prefix $\lambda^0$ of $\sigma$ as follows:
\begin{enumerate}[(i)]
\item if $\sigma_1=1$, let $\lambda^0=1$;
\item if $\sigma_1=2$, let $\lambda^0=\sigma_1\sigma_2=2\sigma_2$;
\item if $\sigma_1\in\{3,4\}$ and $j$ is the smallest integer greater than $1$ such that $\sigma_j\neq 4$, then:
    \begin{enumerate}
    \item if $\sigma_j\in\{1,2\}$, let $\lambda^0=\sigma_1\sigma_2\cdots\sigma_j=\sigma_1\gamma\sigma_j$ for some $\gamma\in\Gamma$;
    \item if $\sigma_j=3$, let $\lambda^0=\sigma_1\sigma_2\cdots\sigma_j\sigma_{j+1}=\sigma_1\gamma 3\sigma_{j+1}$ for some $\gamma\in\Gamma$.
    \end{enumerate}
\end{enumerate}
By construction, $\lambda^0\in\Lambda$. If $\sigma=\lambda^0$, the lemma holds trivially. Otherwise, $\sigma=\lambda^0\sigma'$ for some sequence $\sigma'$ whose last term is $1$ and, by induction hypothesis, there is some $p\geq 1$ and some $\lambda^1,\ldots,\lambda^p\in\Lambda$ such that $\sigma'=\lambda^1\cdots\lambda^p$; thus $\sigma=\lambda^0\lambda^1\cdots\lambda^p$. The uniqueness of $p$ and $\lambda^0,\lambda^1,\ldots,\lambda^p$ follows immediately by induction from the fact that no element of $\Lambda$ is a proper prefix of another element of~$\Lambda$.\end{proof}

From now on, all caterpillars considered in this article have at least two vertices. The following result is an immediate consequence of Lemma~\ref{lem:1}.

\begin{lem}\label{lem:concatenation}
If $T$ is a caterpillar and $s(T)$ is a reduced degree sequence of it, then there exists some integer $p\geq 1$ and some $\lambda^1,\lambda^2,\ldots,\lambda^p\in\Lambda$ such that $s(T)=1\lambda^1\lambda^2\cdots\lambda^p$. Moreover, the integer $p$ and the sequences $\lambda^1,\lambda^2,\ldots,\lambda^p$ are uniquely determined.
\end{lem}

Now, we are ready to introduce a parameter $p(T)$ defined for each caterpillar $T$ and a reduced degree sequence $s(T)$. The \emph{basic sequence number} of a caterpillar $T$, denoted by $p(T)$, is the only positive integer $p$ such that $s(T)=1\lambda^1\lambda^2\cdots\lambda^p$, where $\lambda^i\in\Lambda$ for each $i\in\{1,2,\ldots,p\}$. An example is shown in Figure~\ref{fig:example}.

\begin{figure}
\begin{center}
\includegraphics[scale=.3]{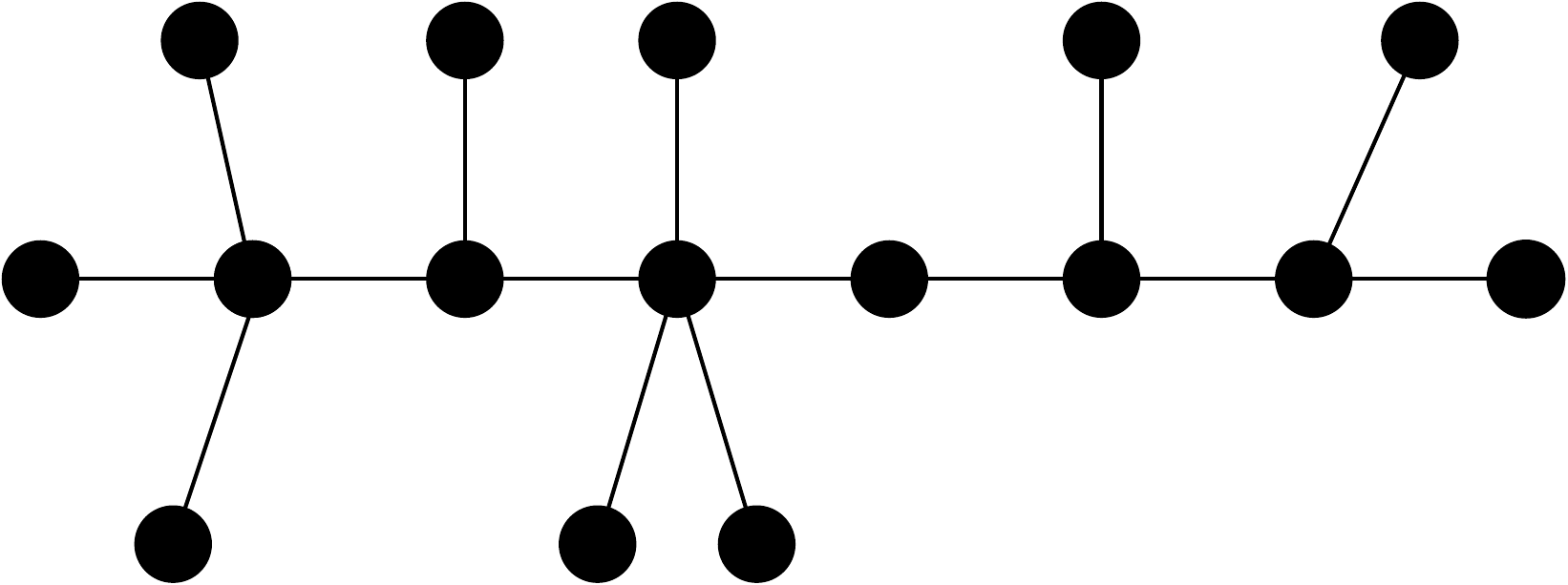}
\end{center}
\caption{The two possible reduced degree sequences of the depicted caterpillar $T$ is $s(T)=14342331$ (resp. $s(T)=13324341$), which is the concatenation of four basic sequences $1$, $434$, $23$, and $31$ (resp. $1$, $332$, $434$, $1$). Thus, $p(T)=4$.}
\label{fig:example}
\end{figure}

\subsection{$P_3$-geodetic number} \label{sec:geodetic_number}
We use $\ell(T)$ (resp.\ $\mathcal{L}(T)$) to denote the number of leaves (resp.\ the set of leaves) of a tree $T$. Recall that a geodetic set of a graph $G$ is a set $S$ of vertices such that every vertex outside $S$ has at least two neighbors in $S$ and $g(G) $ stands for the size of a minimum geodetic set in $G$. 

\begin{theorem}\label{thm: geodetic number of a caterpillar} If $T$ is a caterpillar having at least two vertices, then 
\[ g(T)=p(T)+\ell(T)-1. \]
\end{theorem}
\begin{proof} Let $P=v_1,v_2,\ldots,v_k$ be a dominating path of $T$ where $v_1$ and $v_k$ are leaves of $T$. Let $\sigma=s(T)=d_T(v_1)\tilde d_T(v_2)\cdots\tilde d_T(v_k)$. By Lemma~\ref{lem:concatenation}, $\sigma=1\lambda^1\cdots\lambda^p$ for some $\lambda^1,\ldots,\lambda^p\in\Lambda$, where $p=p(T)$. Let $t_0=1$ and let $t_j=\vert 1\lambda^1\cdots\lambda^j\vert=1+\vert\lambda^1\vert+\cdots+\vert\lambda^j\vert$ for each $j\in\{1,2,\ldots,p\}$. Thus $1=t_0<t_1<\cdots<t_p=k$. 
Let $S$ be the set $\{v_{t_1},v_{t_2},\ldots,v_{t_{p-1}}\}\cup\mathcal{L}(T)$. Since none of $v_{t_1},v_{t_2},\ldots,v_{t_{p-1}}$ is a leaf of $T$, $\vert S\vert=p(T)+\ell(T)-1$.

We claim that $S$ is a geodetic set of $T$. Let $v$ be a vertex in $V(T)\setminus S$. Since $\mathcal L(T)\subseteq S$, $v\in V(P)$. Thus $v=v_h$ for some $h\in\{1,2,\ldots,k\}$. Moreover, since $v_{t_0},v_{t_1},\ldots,v_{t_p}$ are vertices of $S$, it follows that $h\neq t_j$ for each $j\in\{0,1,\ldots,p\}$. Let $j\in\{1,\ldots,p\}$ such that $t_{j-1}<h<t_j$ and let $i$ such that $h=t_{j-1}+i$. Notice that, $\tilde d_T(v_h)=\sigma_h=\lambda^j_i$ for some $1\leq i<\vert\lambda^j\vert$, because $\sigma=1\lambda^1\ldots\lambda^p$. Since $h\notin \{t_0,t_p\}$, $\lambda^j_i\neq 1$. 

Suppose first that $\lambda^j_i=2$. Hence $\lambda^j\in\{21,22,23,24\}$ and $i=1$, because $\lambda^j\in\Lambda$. Consequently $t_{j-1}=h-1$ and $t_j=h+1$. Therefore, $v_h$ is adjacent to $v_{t_{j-1}}$ and $v_{t_j}$. Suppose now that $\lambda^j_i=3$. Hence $\lambda^j\in\Lambda$ and $\lambda^j_1\in\{3,4\}$. Besides, $i=1$ or $i=\vert\lambda^j\vert-1$. Notice that $t_{j-1}=h-1$, $t_j=h+1$, and $d_T(v_h)\ge \tilde d_T(v_h)=\lambda^j_i=3$. On the one hand, if $i=1$, then $v_h$ is adjacent to $v_{t_{j-1}}$ and some leaf of $T$. On the other hand, if $i=\vert\lambda^j\vert-1$, then $v_h$ is adjacent to $v_{t_j}$ and some leaf of $T$. Therefore, $v$ is adjacent to at least two vertices in $S$. It remains to consider the case in which $\lambda^j_i=4$. In this case, $\tilde d_T(v_h)=\lambda^j_i=4$. Hence $v_h$ is adjacent to at least two leaves of $T$. Therefore $v_h$ has at least two neighbors in $S$. This completes the proof of our claim.

Now we must prove that $S$ is a geodetic set of $T$ of minimum cardinality. Consider any geodetic set $S'$ of $T$ and let $j\in\{1,2,\ldots,p-1\}$. We claim that $v_h\in S'$ for some $h\in\{t_{j-1}+1,t_{j-1}+2,\ldots,t_j\}$. Suppose first that $\lambda^j$ has some term equal to $2$. Since $\lambda^j\in\Lambda$, either the first or the last term of $\lambda^j$ is equal to $2$. On the one hand, if the first term of $\lambda^j$ is equal to $2$, then $\tilde d_T(v_{t_{j-1}+1})=2$ and thus $v_h\in S'$ for $h=t_{j-1}+1$ or $h=t_j$. On the other hand, if the last term of $\lambda^j$ is equal to $2$, then $\tilde d_T(v_{t_j})=2$ and thus $v_h\in S'$ for $h=t_j-1$ or $h=t_j$. In both cases the claim holds. We can assume now, without loss of generality, that $\lambda^j$ has no term equal to $2$. Since $j\neq p$, no term of $\lambda^j$ is equal to $1$. Hence each term of $\lambda$ is equal to $3$ or to $4$. Notice that $t_j-2\geq t_{j-1}+1$, because $\vert\lambda^j\vert\geq 3$. Besides, since $\lambda^j\in\Lambda$,  $\tilde d_T(v_{t_j-1})=3$ and thus $v_h\in S'$ for some $h\in\{t_j-2,t_j-1,t_j\}$.  Thus $S'$ contains at least $p-1$ nonpendant vertices of $P$. Since $S'$ contains all leaves of $T$, $\vert S'\vert\ge p+\ell(T)-1=\vert S\vert$. We have already proved that $S$ is a geodetic set of $T$ of minimum cardinality. Therefore $g(T)=p(T)+\ell(T)-1$.\end{proof}

\begin{remark}
	If follows from Theorem~\ref{thm: geodetic number of a caterpillar} that $p(T)$ does not depend on the choice among the two possible linear orders of the vertices of a dominating path of  $T$ (see Fig.~\ref{fig:example}).
\end{remark}

\subsection{$P_3$-hull number} \label{sec:hull_number}
 
Notice that pendant vertices belong to every hull set of any connected graph. 

\begin{lem} \label{lemma}
If $G$ is a graph, then there is a minimum hull set having no vertex of degree $3$ adjacent to a pendant vertex of $G$.
\end{lem}
\begin{proof} Arguing towards a contradiction, suppose that every minimum hull set of some connected graph $G$ has at least one support vertex of degree $3$. Let $S$ be a minimum hull set of $G$ having the minimum possible number $s$ of support vertices of degree $3$ in $G$. Assume that $v_1$ is a support vertex of degree $3$ in $S$, where $N_G(v_1)=\{a,b,c\}$ and $a$ is a pendant vertex of $G$. If $b$ were also a pendant vertex of $G$, then $b\in S$ and $v_1$ would have two pendant vertices, $a$ and $b$, as neighbors. Hence $S\subseteq I[S-\{v_1\}]$ which implies that $S-\{v_1\}$ is a hull set of $G$, contradicting that $S$ is a minimum hull set of $G$. Hence $b$ is not a pendant vertex of $G$. By symmetry, $c$ is not a pendant vertex of $G$. Besides, $b$ and $c$ do not belong to $S$. If $b\in S$, since $a\in S$, then $S\subseteq I_{\mathcal P}[S-v_1]$ and thus $S-\{v_1\}$ would not be a minimum hull set. 
	
Let $v_2=c$. Notice that $S_1=(S-\{v_1\})\cup\{v_2\}$ is also a hull set of $T$. This implies that, if $v_2$ is not a support vertex of degree $3$, then $S_1$ is a minimum hull set with $s-1$ support vertices of degree $3$, contradicting that $s$ is the minimum number of support vertices of degree $3$ for a hull set of $G$. Following this construction, we obtain a path $P=v_1,v_2,\ldots,v_k$ in $G$ such that, for every $i\in\{1,2,\ldots,k\}$, $v_i$ is a support vertex of degree $3$ and $v_k\in S$, where $k$ is the maximum positive integer having this property for some hull set $S$ of $G$. Hence $v_k$ has two neighbors $w_1$ and $w_2$ different from $v_{k-1}$ such that $w_1$ is a pendant vertex and $w_2$ is not a support vertex of degree $3$.  Otherwise, $P'=v_1,\ldots,v_k,w_2$ for $S'=(S-\{v_k\})\cup \{w_2\}$ would be a path longer than $P$ with the same property as $P$.  On the other hand, $w_2$ is not a pendant vertex because otherwise $v_k$ would be adjacent to two pendant vertices and thus $S-\{v_k\}$ would be a minimum hull set having fewer vertices than $s$ support vertices of degree $3$ . Hence $(S-\{v_k\})\cup\{w_2\}$ is a minimum hull set having $s-1$ support vertices of degree $3$, a contradiction. The contradiction arose from supposing that $G$ does not have a minimum hull without support vertices of degree $3$.\end{proof}

\begin{remark}\label{remark1}
Let $G$ be a graph and let $S$ be a hull set. 
\begin{itemize}
\item If $u$ and $v$ are adjacent vertices of degree $2$ of $G$, then $u\in S$ or $v\in S$.
\item Let $u$ be a cut vertex of degree $2$ of $G$ such that $u\notin S$ and whose neighbors are $v$ and $w$. If $G_v$ and $G_w$ are the connected components of $G-u$ such that $v\in V(G_u)$ and $w\in V(G_w)$, then $V(G_v)\cap S$ and $V(G_w)\cap S$ are hull sets of $G_v$ and $G_w$, respectively. 
\end{itemize}

\end{remark}

Let $s$ be a sequence of positive integers. We use $Z(s)$ to denote the summation  $\sum \lfloor { \frac{z_i}{2} } \rfloor$ taken over all the lengths $z_i$ in $s$ of the maximal consecutive subsequences containing only the integer $2$.

\begin{theorem}\label{thm: hull-number-caterpillar}
 Let $T$ be a caterpillar such that $|V(T)|\ge 2$. If $s'(T)$ is the sequence obtained by removing from $s(T)$ those terms corresponding to the integer $3$ (i.e., $s'(T)$ is the subsequence of those terms belonging to $\lbrace1,2,4\rbrace$), then
$$ h(T)=  \ell(T) + Z(s'(T)).$$
\end{theorem}

\begin{proof}

Let $P=v_1,v_2,\ldots,v_k$ be a maximum dominated path of $T$. Let $s(T)=\tilde d_T(v_1),\tilde d_T(v_2),\ldots,\tilde d_T(v_k)$. Recall that $\tilde d_T(v_1)=\tilde d_T(v_k)=1$. Let $B_j= \{v_{j_1},v_{j_2},\ldots,v_{j_{z_j}}\}$ be a maximal set of vertices in $P$ satisfying the following conditions: 

\begin{itemize}
\item $\tilde d_T(v_{j_i})=2$, for every $1 \le i \le z_j$, and
\item either $j_{i+1}=j_i+1$ for every $1\le i \le z_j-1$ or $\tilde d_T(v_h)=3$ for every $j_i < h < j_{i+1}$.  
\end{itemize}
 
Let $J$ be the number of those maximal subsequences in $s(T)$ and $S$ be a minimum hull set of $T$ having no vertex of degree $3$. Lemma~\ref{lemma} guarantees the existence of that minimum hull set. We know that $\mathcal{L}(T) \subset S$. Every vertex $v$ such that $\tilde d_T(v)=4$ does not belong to $S$ because it is adjacent to at least two vertices in $\mathcal{L}(T)$  and thus $v\in I_{\mathcal P}[S]$ for every hull set $S$. We consider the subsets $A_i= \{v_{j_{2i-1}}, v_{j_{2i}}\}$ of $B_j$, for every $1\le i \le z_j/2$ if $z_j$ is even and the subsets $A_i= \{v_{j_{2i-1}}, v_{j_{2i}}\}$ of $B_j$, for every $1\le i\le (z_j-3)/2$ and $A_{\lfloor { \frac{z_i}{2} } \rfloor}= \{ v_{j_{z_j-2}},v_{j_{z_j-1}},v_{j_{z_j}}\}$ if $z_j$ is odd and greater than one. In virtue of the first statement of  Remark~\ref{remark1}, $\vert S \cap A_i \vert\ge 1$ for all $1\le i\le \lfloor { \frac{z_i}{2} } \rfloor$. Hence $\vert S \cap B_j \vert \ge \lfloor \frac{z_j}{2}\rfloor$. This lower bound trivially holds even when $z_j=1$. Therefore, $h(T) \ge  \ell(T) + Z(s'(T))$.

It remains to prove that there exists a hull set $S$ such that $\vert S\vert= \ell(T) + Z(s'(T))$. For every $1\le j\le J$, let $C_j=\{v_{j_2},v_{j_4},\cdots, v_{j_{h_j}} \} \subset B_j$ where $v_{j_{h_j}}= v_{j_{z_j-1}}$ if $z_j$ is odd and $v_{j_{h_j}}= v_{j_{z_j}}$ if $z_j$ is even. Let $S=\mathcal{L}(T) \cup \left( \bigcup_{j=1}^{J} C_j \right)$. It is not hard to prove that $S$ is a hull set.

We have proved that the cardinality of the  minimum hull set is equal to $\mathcal{L}(T)$ plus $\sum_{j=1}^J\lfloor { \frac{z_j}{2} }\rfloor$. Therefore, $ h(T)= \ell(T) + Z(s'(T))$.

\end{proof}

\section{Percolation time}\label{sec: percolation}

\subsection{Percolation time of a caterpillar graph}\label{subsec:PTCG}

In~\cite{1} it was proved that if $T$ is a tree, then $\tau(T)$ can be computed in linear time. In this subsection we give simple closed formula for $\tau(T)$ when $T$ is a caterpillar, in terms of a certain sequences associated to $T$. 

Let $T$ be a caterpillar and let $s(T)$ be a reduced degree sequence of $T$. Let $\ell(i)$ be defined recursively for every $1\le i\le |s(T)|$ as follows: $\ell(1)=1$ and 
\[ \ell(i)=\begin{cases}
             \ell(i-1)+1&\text{if either $\tilde{d}(v_i)\neq 3$, or both $\tilde{d}(v_i)=3$ and $\tilde{d}(v_{i-1})\neq 3$},\\
             \ell(i-1)&\text{otherwise}.
           \end{cases} \]
Let us denote, for every $i\in\{1,\ldots,\vert s(T)\vert\}$, by $n(i)$ (resp.\ $m(i)$) the minimum (resp.\ maximum) integer $j$ such that $\ell(j)=\ell(i)$. In the example considered in Figure~\ref{fig:example}, if $s=1,2,3,4,2,3,3,1,$ $\ell=1,2,3,4,5,6,6,7$, $n=1,2,3,4,5,6,6,8$ and $m=1,2,3,4,5,7,7,8$. We define the sequence $f(T)$, called \emph{percolation sequence of $T$}, as follows. Let $i$ be an integer such that $1\le i\le|s(T)|$. 
\begin{itemize}
	\item $f_i(T)=0$, whenever  $\tilde{d}(v_i)=1$.	
	\item $f_i(T)=\min\{i-n(i),m(i)-i\}+1$, whenever $\tilde{d}(v_i)=3$ and $\tilde{d}(v_{n(i)-1})=\tilde{d}(v_{m(i)+1})\in\{1,2\}$.
	\item $f_i(T)=i-n(i)+1$, whenever $\tilde{d}(v_i)=3$, $\tilde{d}(v_{n(i)-1})=1$ and $\tilde{d}(v_{m(i)+1})=2$.
	\item $f_i(T)=m(i)-i+1$, whenever $\tilde{d}(v_i)=3$, $\tilde{d}(v_{n(i)-1})=2$ and $\tilde{d}(v_{m(i)+1})=1$.
	\item $f_i(T)=\min\{i-n(i)+1,m(i)-i+2\}$, whenever $\tilde{d}(v_i)=3$, $\tilde{d}(v_{n(i)-1})=1$ and $\tilde{d}(v_{m(i)+1})=4$.
	\item $f_i(T)=\min\{i-n(i)+2,m(i)-i+1\}$, whenever $\tilde{d}(v_i)=3$, $\tilde{d}(v_{n(i)-1})=4$ and $\tilde{d}(v_{m(i)+1})=1$.
	\item $f_i(T)=i-n(i)+2$, whenever $\tilde{d}(v_i)=3$, $\tilde{d}(v_{n(i)-1})=4$ and $\tilde{d}(v_{m(i)+1})=2$.
	\item $f_i(T)=m(i)-i+2$, whenever $\tilde{d}(v_i)=3$, $\tilde{d}(v_{n(i)-1})=2$ and $\tilde{d}(v_{m(i)+1})=4$.
	\item $f_i(T)=\min\{i-n(i),m(i)-i\}+2$, whenever $\tilde{d}(v_i)=3$, $\tilde{d}(v_{n(i)-1})=4$ and $\tilde{d}(v_{m(i)+1})=4$.
	\item $f_i(T)=1$ whenever  $\tilde{d}(v_i)=4$.
	\item $f_i(T)=1$, whenever $\tilde{d}(v_i)=2$ and $\tilde{d}(v_{i-1}),\tilde{d}(v_{i+1})\in\{1,2\}$.
	\item $f_i(T)=f_{i+1}(T)+1$, whenever $\tilde{d}(v_i)=2$, $\tilde{d}(v_{i-1})\in \{1,2\}$ and $\tilde{d}(v_{i+1})\notin \{1,2\}$.
	\item $f_i(T)=f_{i-1}(T)+1$, whenever $\tilde{d}(v_i)=2$, $\tilde{d}(v_{i-1})\notin\{1,2\}$ and $\tilde{d}(v_{i+1})\in\{1,2\}$.
	\item $f_i(T)=\max\{f_{i-1}(T),f_{i+1}(T)\}+1$, whenever $\tilde{d}(v_i)=2$, $\tilde{d}(v_{i-1}),\tilde{d}(v_{i+1})\notin\{1,2\}$.
\end{itemize}
We define $F(T)=\max_{1\le i\le r}f_i(T)$. In the example of Figure~\ref{fig:example}, $f(T)=0,1,2,1,3,2,2,1$ and thus $F(T)=3$.
\begin{theorem}
	If $G$ is a caterpillar, then $\tau(T)=F(T)$.
\end{theorem}
\begin{proof}
	Let $P=v_1,v_2,\ldots,v_k$ be a dominating path of $T$ where $v_1$ and $v_k$ are leaves of $T$. Let $s(T)=d_T(v_1)\tilde d_T(v_2)\cdots\tilde d_T(v_p)$. If $p\in\{1,2\}$, the result holds. 
	
	For every $i\in\{1,\ldots,p\}$, if $\tilde{d}(v_i)=1$ or $\tilde{d}(v_i)=4$, then $\tau_T(v_i)=0$ or $\tau_T(v_i)=1$ and thus $\tau(v_i)=f_i(T)$. 
	
	Let $i$ be an integer such that $1<i<k$. Assume that $\tilde{d}(v_i)=3$. Notice that every hull set $S$ of $T$ verifies $\mathcal{L}(T)\subseteq S$. From now on, we are considering percolating set of a given vertex containing $\mathcal{L}(T)$. If $\tilde{d}(v_{n(i)-1})=\tilde{d}(v_{m(i)+1})=1$, then $\mathcal{L}(T)$ is a hull set of $T$ and thus $\tau_S(v_i)\le\tau_{\mathcal{L}(T)}(v_i)$ for every hull set $S$ of $T$. It is easy to prove that $\tau_{\mathcal{L}(v_i)}(T)=\min\{i-n(i)+1,m(i)-i+1\}$ and thus $\tau_T(v_i)=f_i(T)$.  Suppose that $\tilde{d}(v_{n(i)-1})=1$ and $\tilde{d}(v_{m(i)+1})=2$. Notice that if $S$ percolates $v_i$ in $T$ and $\mathcal{L}(T)\subseteq S$, then $R(S)=S-\{v_{n(i)},\ldots,v_{m(i)+1}\}$ also percolates $v_i$. Hence $\tau_S(v_i)\le\tau_{R(S)}(v_i)=i-n(i)+1$ for every hull set $S$ of $T$. Notice that $S=V(T)-\{v_{n(i)},\ldots,v_{m(i)+1}\}$ is a hull set such that $\tau_S(v_i)=i-n(i)+1$. Hence $\tau_T(v_i)=f_i(T)$. Symmetrically, if $\tilde{d}(v_{n(i)-1})=2$ and $\tilde{d}(v_{m(i)+1})=1$, then $\tau_T(v_i)=m(i)-i+1=f_i(T)$. Analogously, if $\tilde{d}(v_{n(i)-1})=1$ and $\tilde{d}(v_{m(i)+1})=4$,  $S$ percolates $v_i$, then $R(S)=S-\{v_{n(i)},\ldots,v_{m(i)+1}\}$ percolates $v_i$, because $\mathcal{L}(T)\subseteq S$. Furthermore, $\tau_{R(S)}(v_i)=f_i(T)=\min\{i-n(i)+1,m(i)-i+2\}$. Hence, using a similar argument to the last case, we conclude that $\tau_T(v_i)=f_i(T)$. Symmetrically, $\tau_T(v_i)=f_i(T)=\min\{i-n(i)+2,m(i)-i+1\}$, if $\tilde{d}(v_{n(i)-1})=4$ and $\tilde{d}(v_{m(i)+1})=1$. The following cases, enumerated below, can be proved following the same line of argumentation and thus their proofs are omitted.
	\begin{itemize}
		\item If $\tilde{d}(v_{n(i)-1})=2$ and $\tilde{d}(v_{m(i)+1})=4$, then $\tau_T(v_i)=i-n(i)+2$.
		\item If $\tilde{d}(v_{n(i)-1})=4$ and $\tilde{d}(v_{m(i)+1})=2$, then $\tau_T(v_i)=m(i)-i+2$.
		\item If $\tilde{d}(v_{n(i)-1})=4$ and $\tilde{d}(v_{m(i)+1})=4$, then $\tau_T(v_i)=\min\{i-n(i)+2,m(i)-i\}+2$.
	\end{itemize} 
	Suppose now that $\tilde{d}(v_{n(i)-1})=2$ and $\tilde{d}(v_{m(i)+1})=2$. Let set $S$ be a set such that percolates $v_i$.  On the one hand, if $v_j\in S$ for some integer $j$ such that $n(i)\le j\le m(i)$, then $\tau_S(v_i)\le |i-j|\le \max\{i-n(i)+1,m(i)-i+1\}$.  On the other hand, if $S$ percolates $v_i$ and does no contain any vertex in $\{v_{n(i)},\ldots,v_{m(i)}\}$, then, by Remark~\ref{remark1}, some of $v_{n(i)-1}$ and $v_{m(i)+1}$ belongs to $S$. In addition, if $v_{n(i)-1}\in S$ (resp. $v_{m(i)+1}\in S$), then $R(S)=S-\{v_{n(i)},\ldots,v_{m(i)+1}\}$ (resp. $R(S)=S-\{v_{n(i)-1},\ldots,v_{m(i)}\}$) also percolates $v_i$. Hence $\tau_S(v_i)\le\tau_{R(S)}=i-n(i)+1$ (resp. $\tau_S(v_i)\le\tau_{R(S)}=m(i)-i+1$). Consequently, $\tau_S(v_i)\le \max\{i-n(i)+1,m(i)-i+1\}$ for every set $S$ that percolates $v_i$. Since $V(T)-\{v_{n(i)},\ldots,v_{m(i)+1}\}$ and $V(T)-\{v_{n(i)-1},\ldots,v_{m(i)}\}$ are hull sets of $T$, $\tau_T(v_i)=\max\{i-n(i)+1,m(i)-i+1\}=f_i(T)$.
	
	It remains to consider the case $\tilde{d}(v_i)=2$. If $\tilde{d}(v_{i-1})\in\{1,2\}$ and $\tilde{d}(v_{i+1})\in\{1,2\}$, by Remark~\ref{remark1} and since $\mathcal{L}(T)\subseteq S$,  $\tau_S(v_i)=1$ for every set $S$ that percolates $v_i$ such that $v_i\notin S$. 
	
	From now on, we are considering a set $S$ percolating $v_i$ such that $\mathcal{L}(T)\subseteq S$. Suppose now that $\tilde{d}(v_{i-1})\notin\{1,2\}$ and $\tilde{d}(v_{i+1})\in\{1,2\}$. Notice that $v_{i+1}\in S$ for every set $S$ percolating $v_i$ such that $v_i\notin S$ because of Remark~\ref{remark1}. If $\tilde{d}(v_{i-1})=4$, then $\tau_S(v_i)\le 2$ for every set $S$ percolating $v_i$ such that $v_i\notin S$. Since $S'=V(T)-\{v_{i+1},v_i\}$ percolates $v_i$ and $\tau_{S'}(v_i)=2$, $\tau_T(v_i)=2$. Symmetrically, if $\tilde{d}(v_{i+1})=4$ and $\tilde{d}(v_{i-1})\in\{1,2\}$, then $\tau_T(v_i)=f_i(T)$. Assume now that $\tilde{d}(v_{i-1})=3$ and $\tilde{d}(v_{i+1})\in\{1,2\}$. Hence, if $S$ is a hull set and thus percolates $v_i$ and $v_i\notin S$, then $v_{i+1}\in S$ (see Remark~\ref{remark1}). Besides, if $\mathcal{L}(T)\subseteq S$ then $S$ percolates $v_{n(i)-1}\in S$ whenever $\tilde{d}(v_{n(i)-1})\notin\{1,2\}$. Hence $R(S)=S-\{v_{n(i)-1},\ldots,v_{i+1}\}$  also percolates $v_i$ when $\tilde{d}(v_{n(i)-1})=4$, and $R(S)=S-\{v_{n(i)},\ldots,v_{i+1}\}$ also percolates $v_i$ otherwise.  
 	Thus $\tau_S(v_i)\le\tau_{R(S)}=f_i(T)+1$ for every $S$ percolating $v_i$. In addition, since $S'=V(T)-\{v_{n(i)-1},\ldots,v_i\}$ is a hull set of $T$ if $\tilde{d}(v_{n(i)-1})=4$ and $S'=V(T)-\{v_{n(i)},\ldots,v_i\}$ is a hull set of $T$ otherwise, and $\tau_{S'}(v_i)=f_i(T)+1$, it follows that $\tau_T(v_i)=f_i(T)+1$. By symmetry,  if $\tilde{d}(v_{i-1})\in\{1,2\}$ and $\tilde{d}(v_{i+1})\notin\{1,2\}$, then $\tau_T(v_i)=f_{i+1}(T)+1$. Following a similar line of argumentation it can be proved that if $\tilde{d}(v_{i-1}),\,\tilde{d}(v_{i+1})\notin\{1,2\}$, then $\tau_T(v_i)=\max\{f_i(T),f_{i+1}(T)\}+1$. 
 	
 	Since $\tau(T)$ is the maximum $\tau_S(v)$ among all vertices in $T$ and hull sets $S$ of $T$, the result follows from the above discussions.
\end{proof}

\subsection{Unit interval graphs}\label{subsec: UIG}

Let $(G,<)$ be linear order of the vertices of a graph $G$ such that $N[v]$ is an interval in that order, for every $v\in V(G)$. Denote by $v_L$ and $v_R$ the minimum and maximum vertex under that order, respectively.


\begin{prop}\label{shortest-path}
Let $G$ be a unit interval graph with a unit interval order $(G,<)$. If $u<v$, then every shortest path $u=y_1,\ldots,y_r=v$ verifies that $y_i<y_{i+1}$ for each integer $i$ such that $1\le i\le r-1$
\end{prop}

\begin{proof}
Consider a shortest path $u=y_1,\ldots,y_r=v$ in $G$. Since $P$ is a shortest path, $P$ turns out to be an induced path of $G$. Suppose, towards a contradiction, that there exists an integer $j$ such that $y_j<u$ and let $i$ be the biggest integer $i$ such that $y_i<u$. Since $y_i<u<y_{i+1}$, $u$ is adjacent to $y_{i+1}$, contradicting that $P$ is an induced path. Hence $u< y_i$ for every $1< i\le r$.  Analogously, it can be proved that $y_i<v$ for every $1\le i< r$. Now, suppose towards a contradiction that $y_{j+1}<y_j$ for some integer $j$ such that $1< j< r$. Since $y_j<y_r=v$, there exists an integer $\ell$ such that $j<j+1\le \ell<\ell+1\le r$ such that $y_{\ell}<y_j<y_{\ell+1}$. Hence $y_j$ is adjacent to $y_{\ell+1}$, contradicting that $P$ is an induced path path of $G$. 
\end{proof}

\begin{prop}\label{prop: distance uig}
Let $G$ be a connected unit interval graph with a unit interval order $(G,<)$. If $u<v<w$, then $d(u,v)\le d(u,w)$.
\end{prop}

\begin{proof}
Let $P=u=y_1,\cdots,y_r=w$ be a shortest path. By Proposition~\ref{shortest-path}, there exist an integer $j$ such that $1< j <r$ and $y_j\le v< y_{j+1}$. If $y_j=v$, then $d(u,v)=j-1<r-1=d(u,w)$. Assume now that $y_j<v< y_{j+1}$ and thus $v$ is adjacent to $y_j$. Since $u=y_1,\ldots,y_j,v$ is a path, $d(u,v)\le j\le r-1=d(u,w)$. 
\end{proof}

Denote by $a_R(v)$ (resp.\ $a_L(v)$) to the rightmost (resp.\ leftmost) adjacent vertex of $v$. 

\begin{coro}\label{cor: distance uig}
Let $G$ be a unit interval graph. 
\begin{enumerate}
	\item If $P\mbox{: }u=v_1,\cdots,v_k=w$ is a path such that $v_{i+1}=a_R(v_i)$ for every $1\le i\le k-1$, then $d(u,w)=k-1$; i.e, $P$ is a shortest path between $u$ and $w$.
	\item If $P\mbox{: }u=v_1,\cdots,v_k=w$ is a path such that $v_{i+1}=a_L(v_i)$ for every $1\le i\le k-1$, then $d(u,w)=k-1$; i.e, $P$ is a shortest path between $u$ and $w$.
\end{enumerate}
  
\end{coro}

\begin{proof}
Let $P'\mbox{: }u=v'_1,\cdots,v'_s=w$ be a shortest path; i.e, $s-1=d(u,w)$. Clearly, $v'_2\le a_R(v_1)=v_2$. Suppose, by inductive hypothesis, that $v'_h\le v_h$ for every integer $h$ such that $1< h< k \le s$. We will prove that $v'_k\le v_k$. Suppose, towards a contradiction, that $v'_k> v_k$. By inductive hypothesis, $v'_{k-1}\le v_{k-1}<v_k<v'_k$. Since $v'_{k-1}$ is adjacent to $v'_k$, $v_{k-1}$ is adjacent to $v'_k$. Hence $a_R(v_{k-1})=v_k<v'_k\le a_R(v_{k-1})$, a contradiction. This contradiction arose from supposing that $v'_k> v_k$.  Consequently $v'_k\le v_k$. We have already proved by induction that $v'_i\le v_i$ for every $1\le i\le s$ which implies $k\le s$ and thus, since $P'$ is a shortest path connecting $u$ and $w$, $d(u,w)=k-1$ as we want to prove. Suppose, towards a contradiction, that $s<k$ and thus $v_s<w$. Hence  $v'_{s-1}\le v_{s-1}<v'_s=w$ which implies that $v_{s-1}$ is adjacent to $w$. Thus $v_s=a_R(v_{s-1})\ge w$, contradicting that $v_s<w$. Therefore, $k\le s$. The second statement is proved in a symmetric way.
\end{proof}

By combining Proposition~\ref{prop: distance uig} and Corollary~\ref{cor: distance uig}, we obtain the following result.

\begin{coro}\label{Corolario-diametro}
If $G$ is a connected unit interval graph, then $d=\mathrm{diam}(G)=d(v_L,v_R)$. Besides, $P^L\mbox{: }v_L=v_1,a_R(v_1),\ldots,a_R(v_{k-1})=v_R$ and \\$P^R\mbox{: }v_R=v_1,a_L(v_1),\ldots,a_L(v_{k-1})=v_L$ are diameter paths of $G$.
\end{coro}

\subsubsection{Percolation time of unit interval graph}\label{PT-UIG}

If $G$ is a unit interval graph, we say that a subset $S\subseteq V(G)$ is an \emph{interval respect to that order} if all vertices in $S$ appears consecutively. Notice that cliques in $G$ are intervals. The following result is a very helpful tool for the rest of the section.

\begin{theorem}[\cite{Centeno2011}]\label{thm: centeno2011}
Let $G$ be a $2$-connected chordal graph. If $u$ and $v$ are two vertices with at least one vertex in common, then $H=\{u,v\}$ is a $P_3$-hull set of $G$.
\end{theorem}

\begin{lem}\label{lema_de_los_extremos}
Let $G$ be a $2$-connected unit interval graph. If $H=\{u,v\}$ is a hull set of $G$ such that $H\neq \{v_L,v_R\}$, then $\tau_H(G)=\max\{\tau_H(v_L),\tau_H(v_R)\}$. 
\end{lem}
 
\begin{proof}
Throughout the proof we will consider that the vertices of $G$ have a unit interval order ``$<$''. Assume that $u<v$. We will split the proof into two claims. 

\medskip
\textbf{Claim 1: } If $I^k[H]$ is an interval for some integer nonnegative integer $k$, then $I^{k+	1}[H]$ is an interval.

\medskip
Assume that $I^k[H]=[a,b]$ for some nonnegative integer $k$. Hence $A=I^{k+1}[H]\setminus I^{k}[H]=\bigcup C$, where the union is taking over all cliques $C$ such that $\vert C\cap [a,b]\vert\ge 2$; notice that, since $H$ is a hull set, $A\neq\emptyset$ whenever $I^k[H]\neq V(G)$.   Since $A$ is the union of intervals having nonempty intersection with $I^k[H]$ which is an interval, it follows that $I^{k+1}[H]=A\cup I^k[H]$ is also an interval.

\medskip
\textbf{Claim 2: } $I^2[H]$ is an interval.

\medskip
First, suppose that $uv\notin E(G)$. Notice that if $x\in I^1[H]\setminus H$, then $u<x<v$, because otherwise $u$ would be adjacent to $v$. In addition, $H_1=I^1[H]\setminus H$ is a clique. Suppose, towards a contradiction, that there exists two nonadjacent vertices $x,\, y\in H_1$. Assume, without loss of generality, that $x<y$. Hence, since $x<y<v$ and $xv\in E(G)$, $x$ is adjacent to $y$, reaching a contradiction. Since $H_1$ is a clique, it turns out to be an interval $[u_L,u_R]$ under the considered unit interval order. If there is no vertex $x$ such that either $u<x<u_L$ or $u_R<x<v$, then $I^1[H]$ is an interval and the assertion follows by Claim 1. Suppose now that it is not the case, and there exists a vertex $x\notin I^1[H]$ such that $u<x<u_L$. Since $u$ is adjacent to $u_L$,  $x$ is adjacent to $u$ and $u_L$. Analogously, if there exists a vertex $x$ such that $u_R<x<v$, then $x$ is adjacent to $u_R$ and $v$. Thus $[u,v]\subseteq I^2[H]$. In addition, if $z<y<u$ and $z\in I^2[H]$, then $y\in I^2[H]$. Because $z$ is adyacent to some vertex $x\in H_1$ and thus $z$ is also adjacent to $u$, which implies that $y$ is adjacent to $u$ and $x$. Symmetrically, if $v<y<z$ and  $z\in I^2[H]$, then $y\in I^2[H]$. Therefore, $I^2[H]$ is an interval.

Finally, suppose that $u$ is adjacent to $v$. Hence $x\in H_1$ if and only if $R=\{x,u,v\}$ is contained in a clique $C$ in $G$. Therefore $I^1[H]=\bigcup C $ where the union is taking over all the cliques containing $H$. Since all such cliques are intervals, $I^1[H]$ turns out to be an interval which implies, by Claim 2, that $I^2[H]$ is also an interval. 

\medskip
On the one hand, since $H\neq\{v_L,v_R\}$, if $\tau_H(G)=1$, then the result holds. On the other hand, if $k\ge 2$, since $I^k[H]$ is an interval for every $k\ge 2$ because of Claims 1 and 2, either $v_L\notin I^k[H]$ or $v_R\notin I^k[H]$ for every $1\le k<\tau_H(G)$. The result follows from these observations. 
\end{proof}

Let``$<$'' be a unit interval order of a unit interval graph graph $G$. We use $v\downarrow$ and $v\uparrow$ to denote the vertex immediately before $v$ and the vertex immediately after $v$, respectively. We denote by $V_k(H)$ the set of vertices with percolation time $k$ under the hull set $H$. We denote $L(I)$ and $R(I)$ to the leftmost vertex and the  rightmost vertex of the set $I\subseteq V(G)$ under the order ``$<$''.

\begin{coro}\label{dos-vertices-extremos}
If $G$ is a unit interval graph such that $|V(G)|\ge 3$ every two maximal cliques with nonempty intersection have at least two common vertices, then $\tau(\{v_L,v_L{\uparrow} \})=\tau(\{v_R{\downarrow},v_R\})=\mathrm{diam}(G)=\tau(G)$. 
\end{coro}

\begin{proof}
Notice that $G$ is $2$-connected. Otherwise, $G$ would have a cut vertex $v$ and thus, since $G$ is chordal, $\{v\}=C\cap C'$ where $C$ and $C'$ are maximal cliques of $G$ (see~\cite{McKee}). In addition, two maximal cliques with nonempty intersection have at least two common vertices and $|V(G)|\ge 3$, the only maximal clique containg $v_L$ (resp. $v_R$) contains at least three vertices. Hence, by Theorem~\ref{thm: centeno2011}, $\{v_L,{\uparrow} v_L\}$ (resp. $\{{\downarrow} v_R,v_R\}$) is a hull set of $G$

Since $G$ is a $2$-connected graph, it is easy to see, by applying Theorem~\ref{thm: centeno2011}, that $\tau(G)=1$ if and only if $G$ is a complete graph. From now on, we will assume that $\tau(G)\ge 2$. Let $H$ be a hull set of $G$. We can assume by Theorem~\ref{thm: centeno2011} and Lemma~\ref{lema_de_los_extremos}, without loss of generality, that $H=\{u,v\}$ for two vertices $u$ and $v$ having a common neighbor, and $\tau(G)=\tau_H(G)=\tau_H(v_R)$. Assume, without loosing generality, that $u<v$, and thus $v_L\le u$ and $v_L{\uparrow}\le v$. Hence $R(I^1[\{v_L,v_L{\uparrow}\}])\le R(I^1[H])$. Consequently, it can be proved by induction that  $R(I^k[\{v_L,v_L{\uparrow}\}])\le R(I^k[H])$ for every $1\le k\le\tau(G)$. Thus $\tau(G)=\tau_{H}(v_R)\le \tau_{\{v_L,v_L{\uparrow}\}}(v_R)$. Therefore $\tau(G)=\tau(\{v_L,v_L{\uparrow}\})$. By symmetry, if $\tau(G)=\tau_H(G)=\tau_H(v_L)$, then $\tau(G)=\tau(\{v_R{\downarrow},v_R\})$.

Notice that $Q_k=[a_L(R(I^k[H]])),R(I^k[H])]$ is a maximal clique. Otherwise, would there exist a maximal clique $Q=[x,y]$ containing $Q_k$ with $R(I^k[H])<y$, since $Q_k$ contains at leas two vertices in $I^{k-1}[H]$, $y\in I^k[H]$. Since there is no two maximal cliques with nonempty intersection having at most one common vertex, $a_R(R[I^k[H]])=R[I^{k+1}[H]]$. Otherwise, would there exist a vertex $x$ adjacent to $R(I^k[H])$ and nonadjacent to $R(I^k[H]){\downarrow}$ and thus the maximal clique $[R[I^k[H]],a_R(R(I^k[H]))]$ would be a maximal clique having exactly one vertex in common with the maximal clique $[a_L(R(I^k[H]])),R(I^k[H])]$, which is precisely the vertex $R(I^k[H])$. By Corollary~\ref{Corolario-diametro}, the path $v_L,R(I^1[H]),\ldots,R(I^{\tau(G)}),v_R$ is a diameter path. By symmetry,  if $H'=\{v_R{\downarrow},v_R\}$, then the path $v_R,L(I^1[H']),\ldots,L(I^{\tau(G)}),v_L$ is also diameter path.  Finally, by Lemma~\ref{lema_de_los_extremos} and the previous discussion, we conclude that $\tau(G)=\mathrm{diam}(G)=\tau(\{v_L,v_L{\uparrow} \})=\tau(\{v_R{\downarrow},v_R\})$.
\end{proof}

From the first part of the proof of Corollary~\ref{dos-vertices-extremos} can be derived the following remark.

\begin{remark}\label{remark2}
	If $G$ is a $2$-connected unit interval graph such that $|V(G)|\ge 3$, then either $\tau(G)=\tau(\{v_L,v_L{\uparrow}\})$ or $\tau(G)=\tau_G(\{v_R{\downarrow},v_R\})$.
\end{remark}

\begin{theorem}\label{quliques-con-un-solo-vertices-en-comun}
If $G$ is a 2-connected unit interval graph with such that $|V(G)|\ge 3$, then there exists a 2-connected unit interval graph $G^*$ such that every two nonempty-intersection maximum clique has at least two common vertices and $\tau(G)=\tau(G^*)=\mathrm{diam}(G^*)$. Besides, $\tau(G)=\tau_G(\{v_1,v_2\})=\tau_G(\{v_{n-1},v_n\})$.
\end{theorem}

\begin{proof}
 Since $G$ is $2$-connected and does not contain any induced cycle with at least four vertices, every maximal clique of $G$ has at least three vertices. Indeed, if $C=[a,b]$ is clique with two vertices, then $ab$ is a bridge of $G$, an edge whose removal disconnect the graph, an thus, since $|V(G)|\ge 3$, $a$ or $b$ is a cut vertex, contradicting that $G$ is $2$-connected. Hence,  $ab$ is and edge of a complete graph on three vertices. Let $v_1,v_2,\ldots,v_n$ be a unit interval order for $V(G)$. Recall that a maximal clique of $G$ is a maximal interval of pairwise adjacent vertices. Besides, two vertices are adjacent if and only if they belong to at least one of these maximal intervals of pairwise adjacent vertices. In other words, these intervals define the adjacencies of $G$, once the unit interval order of $|V(G)|$ was established. We are going to define the maximal intervals of a graph $G^*$ obtained from $G$ by properly adding some vertices without modifying the relative order of those vertices in $V(G)$.

First, set $w_i=v_i$ for each $1\le i\le n$.  We will traverse the vertices of $G$ from $w_1$ to $w_n$. Set $G=G_1$; and any time we find a vertex $w_h$ such that $\{w_h\}=C\cap C'$, where $C$ and $C'$ are maximal cliques of $G_i$, apply the following transformation; we will call such vertices singular vertices. Add a vertex $w_{n+k}$ at the end, where  $k-1$ stands for the number of singular vertices  previously processed, one for each graph $G_j$ for every $1\le j\le i-1$, and replace each clique $C=[a,b]$ by $C'=[a,b]$ if $b\le w_h$, by $C'=[a,b{\uparrow}]$ if $a< w_h< b$, by $C'=[a{\uparrow},b{\uparrow}]$ if $a\ge w_h$.  Notice that if $b=w_{n+k-1}$, then $b{\uparrow}=w_{n+k}$. The new cliques define a unit interval graph $G^*$ with vertex set $\{v_1\ldots,v_{n+m}\}$, where $m$ is the number singular vertices of $G$. Notice that, the number of maximal cliques of $G$ and $G^*$ agree. If $C_1,\ldots, C_r$ denote the set of maximal cliques of $G$, we denote by $C_1^*,\ldots,C_r^*$ the maximal cliques obtained by the transformation of $G$ into $G^*$, where $C_i^*$ is the clique corresponding to $C_i$ for every $1\le i\le k$, under the previously described transformation. Besides, the relative order of the left endpoints (resp. right endpoints) of the maximal cliques of $G$ is not modified when transforming $G$ into $G^*$, and $C_i^*\cap C_j^*\neq\emptyset$ if and only if $|C_i\cap C_j|\ge 2$.

It is easy to prove that two nonempty intersecting cliques in $G^*$ has at least two vertices in common. In addition, $G^*$ is $2$-connected. Since each pair of intersecting cliques has at least two vertices in common, it suffices to prove that $G^*$ is connected. We are going to proceed to prove it by induction. Recall that, if $G^*$ has a cut vertex $v$, since $G^*$ is chordal~\cite{McKee}, then $\{v\}=C\cap C'$ where $C$ and $C'$ are maximal cliques of $G^*$. Suppose, towards a contradiction, that $G^*$ is disconnected. Hence there exists a positive integer $i$ such that $1< i< n+m$ in $G^*$, such that every vertex $v$ of $G^*$ such that $v\le w_i$ is nonadjacent to every vertex  $w$ of $G^*$ such that $w_{i+1}\le w$. Thus there exists an index $j$ such that the maximal clique $[w_{i+1},a_R(w_{i+1})]$ of $G_{j+1}$ comes from the maximal clique $[w_i,a_R(w_i)]$ and thus $w_{i+1}=w_i{\uparrow}$, $a_R(w_{i+1})=a_R(w_i){\uparrow}$ and $w_i$ is a singular vertex of $G_j$. Since  $G_j$ is $2$-connected, by inductive hypothesis, and thus $w_i$ is not a cut vertex of $G_j$, there exists a maximal clique $[x,y]$ in $G_j$ such that $a_L(w_j)<x<w_j$ and $w_j<y<a_R(w_j)$. By construction, $x<w_j$ in $G_{j+1}$, $w_{j+1}<y{\uparrow}$ in $G_{j+1}$, and $x$ is adjacent to $y{\uparrow}$ in $G_{j+1}$, a contradiction. Therefore, $G_{j+1}$ is connected.

It remains to prove that $\tau(G^*)=\tau(G)$. We consider $C_1,\ldots,C_s$ and $C_1^*,\ldots,C_s^*$ the maximal of $G$ and $G^*$ respectively ordered by their left endpoints. Since $G$ and $G^*$ are $2$-connected, it suffices to prove that $\tau_{G^*}(\{w_1,w_2\})=\tau_G(\{v_1,v_2\})=\tau_G(\{v_{n-1},v_n\})=\tau_{G^*}(w_{n+m-1},w_m)$ (see Corollary~\ref{dos-vertices-extremos} and Remark~\ref{remark2}). In order to prove it, we have to note that $I^k[\{v_1,v_2\}]=I[[v_1,v_{k-1}]]=[v_1,v_k]$ and  $I^k[\{w_1,w_2\}]=I[[w,w_{k-1}]]=[w_1,w_k]$ , where $v_k=a_R(v_{k-1}{\downarrow})$ and $w_k=a_R(w_k)=a_R(w_{k-1})$, and in addition, there exists an integer $i$ such that $1< i\le s$ and  if $C$ and $C^*$ are the maximal cliques of $G$ and $G^*$ respectively, containing $a_R(v_{k-1}{\downarrow})$ and $a_R(w_k)$ as their leftmost endpoint, then $C=C_i$ and $C^*=C_i^*$.  Therefore, $\tau_{G^*}(\{w_1,w_2\})=\tau_G(\{v_1,v_2\})$. Recall that the relative order of the right endpoints and the left endpoints of the corresponding maximal cliques of $G^*$ did not suffer any modifications respect to the maximal cliques of $G$. Therefore, by symmetry, $\tau_G(\{v_{n-1},v_n\})=\tau_{G^*}(\{w_{n+m-1},w_{n+m}\})$.
\end{proof}

Now we have to consider the case in which $G$ has at least one cut vertex.

\begin{prop}\label{vertices-de-corte}
Let $G$ be a connected graph. If $C_1,\ldots,C_r$ are the $2$-connected components of $G$, then $\tau(G)\le\sum_{i=1}^r \tau(C_i)$.
\end{prop}

\begin{proof}
Let $S$ be a hull set of $G$. Denote by $s_i$ the  minimum $h>0$  such that $V(C_i)\cap V_h[S]$ is a hull set of $C_i$; and denote by $t_i$ the minimum $h$ such that $V(C_i)\subseteq I^h[S]$. 
\[\tau(G)\le \sum_{i=1}^r (t_i-s_i+1)\le \sum_{i=1}^r \tau(C_i).\qedhere \]  
\end{proof}

\begin{theorem}\label{thm: percol-sin-vertexcut of degree two}
If $G$ is a connected unit interval graph with no cut vertex of degree $2$ and no vertex of degree $1$, then $\tau(G)=\sum_{i=1}^r\mathrm{diam}(C_i^*)$. In addition, $\tau(G)=\tau(v_L)=\tau(v_R)$
\end{theorem}

\begin{proof}
Since $G$ has no vetices of degree $1$, $d(v_1)\ge 2$ and $d(v_n)\ge 2$. Each $2$-connected component is an interval of consecutive vertices in the unit interval order $v_1,\ldots,v_n$ of $G$, where both endpoints are cut vertices but $v_1$ and $v_n$; i.e., $V(C_i)=[v_{m_{i-1}},v_{m_i}]$ for every $1\le i\le r$, where $m_i=\sum_{j=1}^i\vert V(C_j)\vert-i+1$ ($m_0=1$). Since $G$ has no cut vertex of degree $2$, it is easy to see, using Theorem~\ref{quliques-con-un-solo-vertices-en-comun}, that $S=\{v_1,v_2,v_{m_1+1},\cdots,v_{m_{r-1}+1}\}$ (resp. $\{v_n,v_{n-1},v_{m_{r-1}-1},\ldots,v_{m_1-1}\}$) is a hull set, and $\tau_G(S)=\sum_{i=1}^r \tau(C_i)=\tau(v_n)$ (resp. $\tau_G(S)=\sum_{i=1}^r \tau(C_i)=\tau(v_1)$). Therefore, by  Proposition~\ref{vertices-de-corte}, $\tau(G)=\sum_{i=1}^r\mathrm{diam}(C_i^*)=\tau(v_L)=\tau(v_R)$.
\end{proof}
%
%
%
%
Let $G$ be a connected unit interval graph whose $2$-connected component are $G_1,\ldots,G_r$ and $x\le y$ for every $x\in V(G_i)$ and $y\in V(G_j)$ such that $1\le i< j\le r$, we define $G^*$ as the unit interval graphs whose $2$-connected components are $G_1^*,\ldots,G_r^*$ and $x\le y$ for every $x\in V(G_i^*)$ and $y\in V(G_j^*)$ such that $1\le i< j\le r$. Notice that in both cases the equality holds when $x=y$ is a cut vertex.  In addition, if $G$ has no vertex of degree $1$ and it does not have cut vertices of degree $2$, by Theorem~\ref{thm: percol-sin-vertexcut of degree two}, $\tau(G)=\tau(G^*)=\sum_{i=1}^r\mathrm{diam}(G_i^*)$.

Let $G$ be a connected unit interval graph with at least three vertices having a unit interval order of its vertices $v_1,\ldots,v_n$. If $u$ is a vertex of degree $1$, since $G$ is connected, $v=v_1$ or $v=v_n$. Besides, if $u$ is a cut vertex of degree $2$ whose only neighbors are $v$ and $w$, then $v$ is nonadjacent $w$ and  either $v<u<w$ or $w<u<v$. We will call an \emph{special subgraph} to a graph induced subgraph by those vertices in an maximal interval $[a,b]$ such that $2<d(v)$ for every cut vertex $v$ of $G$ such that $a<v<b$ and also $a$ or $b$ belongs to $\{v_1,v_n\}$, or it is a cut vertex of degree $2$ of $G$. Such interval $[a,b]$ will be called \emph{special interval}.

If $H$ is an special subgraph of $G$ induced by $[a,b]$ we define $t(H)$ as follows: 
\begin{itemize}
	\item $t(H)=1$, if $|V(H)|= 2$.
	\item $t(H)=\tau(H^*)$, if $v_1=a$, $v_n=b$, $2\le d(v_1)$ and $2\le d(v_n)$.
	\item $t(H)=\tau((H-a)^*)+1$, if [($v_1=a$ and $d(v_1)=1$) or $v_1<a$], $v_n=b$,  and $2\le d(v_n)$.
	\item $t(H)=\tau((H-b)^*)+1$, if $v_1=a$, $2\le d(v_1)$ and [($d(v_n)=1$ and $v_n=b$) or $b<v_n$].
	\item $t(H)=\tau((H-\{a,b\})^*)+1$, if $v_1=a$, $v_n=b$, $d(v_1)=d(v_n)=1$.
	\item $t(H)=\tau((H-\{a,b\})^*)+2$ if ($v_1<a$, $v_n=b$, and $d(v_n)=1$), or ($v_1=a$, $b<v_n$, and $d(v_1)=1$), or ($v_1<a$ and $b<v_n$).
\end{itemize} 
 We define $\epsilon(H)$ as the maximum $t(H)$ among all the special subgraphs $H$ of $G$.
\begin{theorem}
If $G$ is a connected unit interval graph such that $|V(G)|\ge 3$, then $\tau(G)=\epsilon(G)$. 
\end{theorem}
\begin{proof}
Suppose first that $G$ has no cut vertex of degree $2$. If $G$ has no vertex of degree $1$, the result follows by Theorem~\ref{thm: percol-sin-vertexcut of degree two}. Suppose now that $G$ has at least one vertex of degree $1$. Such a vertex could be either $v_1$ or $v_n$. Assume that $d(v_1)=1$ and $d(v_n)\ge 2$. Hence $v_1\in S$ for every hull set $S$ of $G$. Since $G-v_1$ is a $2$ connected unit interval graph and it has neither cut vertex of degree $2$ no vertex of degree $1$, if $S-v_1$ is a hull set of $G-v_1$, then $\tau_S(G)= \tau_{S-v_1}(G-v_1)\le\tau((G-v_1)^*)<t(G)$. Assume that $S-v_1$ is not a hull set of $G-v_1$. Thus there exist a vertex $u\in V(G-v_1)$ such that $v_1$ and $u$ has a common neighbor in $G$, this vertex only can be $v_2$; i.e, $u=v_2$, which implies that $\{v_2,v_3\}\subseteq I_{P_3}^1[S]$. Consequently, by Theorem~\ref{thm: percol-sin-vertexcut of degree two}, $\tau_S(G)\le \tau((G-v_1)^*)+1=t(G)$. Since $S=\{v_1,v_3\}$ is a hull set and $I_{P_3}^1[S]-\{v_1\}=\{v_2,v_3\}$, it follows that $\tau_S(G)=\tau((G-v_1)^*)+1$. Therefore, $\tau(G)=t(G)$.  Symmetrically, if $d(v_1)\ge 2$ and $d(v_n)=1$, then $\tau(G)=t(G)$. Following the same line of argumentation, it can be proved that the result also holds if $d(v_1)=d(v_n)=1$. That is why the details are omitted.

Suppose now that $G$ has at least one cut vertex $v$ of degree $2$ and let $[a,b]$  a special interval such that $|[a,b]|\ge 3$ and $H=G[[a,b]]$. Assume that one of $a$ and $b$ is a cut vertex of degree $2$ in $G$, say $a$. Using the similar techniques to those of the above paragraph, it can be proved that: 
\begin{itemize}
	\item $\tau(H-a)=\max\{\tau_G(a{\uparrow}),\tau_G(b)\}=\max\{\tau((H-a)^*),\tau((H-a)^*)+1\}=\tau((H-a)^*)+1$, whenever $d(b)\ge 2$,
	\item $\tau(H-\{a,b\})=\max\{\tau_G(a{\uparrow}),\tau_G(b{\downarrow})\}=\max\{\tau((H-\{a,b\})^*)+1,\tau((H-\{a,b\})^*)+1\}=\tau((H-\{a,b\})^*)+1$, whenever $d(b)=1$ or $b$ is cut vertex of degree $2$.
\end{itemize}

Analogously, if $b$ is a cut vertex of degree $2$, then the following two conditions hold.

\begin{itemize}
	\item $\tau(H-b)=\max\{\tau_G(a),\tau_G(b{\downarrow})\}=\max\{\tau((H-b)^*)+1,\tau((H-b)^*)\}=\tau((H-b)^*)+1$, whenever $d(a)\ge 2$,
	\item $\tau(H-\{a,b\})=\max\{\tau_G(a{\uparrow}),\tau_G(b{\downarrow})\}=\max\{\tau((H-\{a,b\})^*)+1,\tau((H-\{a,b\})^*)+1\}=\tau((H-\{a,b\})^*)+1$, whenever $d(a)=1$ or $a$ is a cut vertex of degree $2$.
\end{itemize}

The result follows by combining this facts with the following: if $v$ is a cut vertex of $G$ such that $d(v)=2$ then $\tau(v)=\max\{\tau(v{\downarrow}),\tau(v{\uparrow})\}$.   
\end{proof}

\section{A special graph class}\label{sec:special graph}


A graph $G$ satisfies the property $\mathcal P$ if $I^{2}_{P_3}[S]=I_{P_3}[S]$ for every set $S \subset V(G)$. It is easy to see that this property is hereditary and every graph $G$ belonging to this class satisfy $\tau(G)\le 1$. The next result characterizes, by minimal forbidden induced subgraphs, those graphs satisfying property $\mathcal P$.

\begin{theorem}\label{thm: special class}
Let $G$ be a graph. Then, $G$ satisfies the property $\mathcal P$ if and only if it does not contain as induced subgraph any graph depicted in Figure~\ref{fig:prohibited_graph}.
\end{theorem}

\begin{figure}[htb] 
\begin{center}
\includegraphics[scale=0.5]{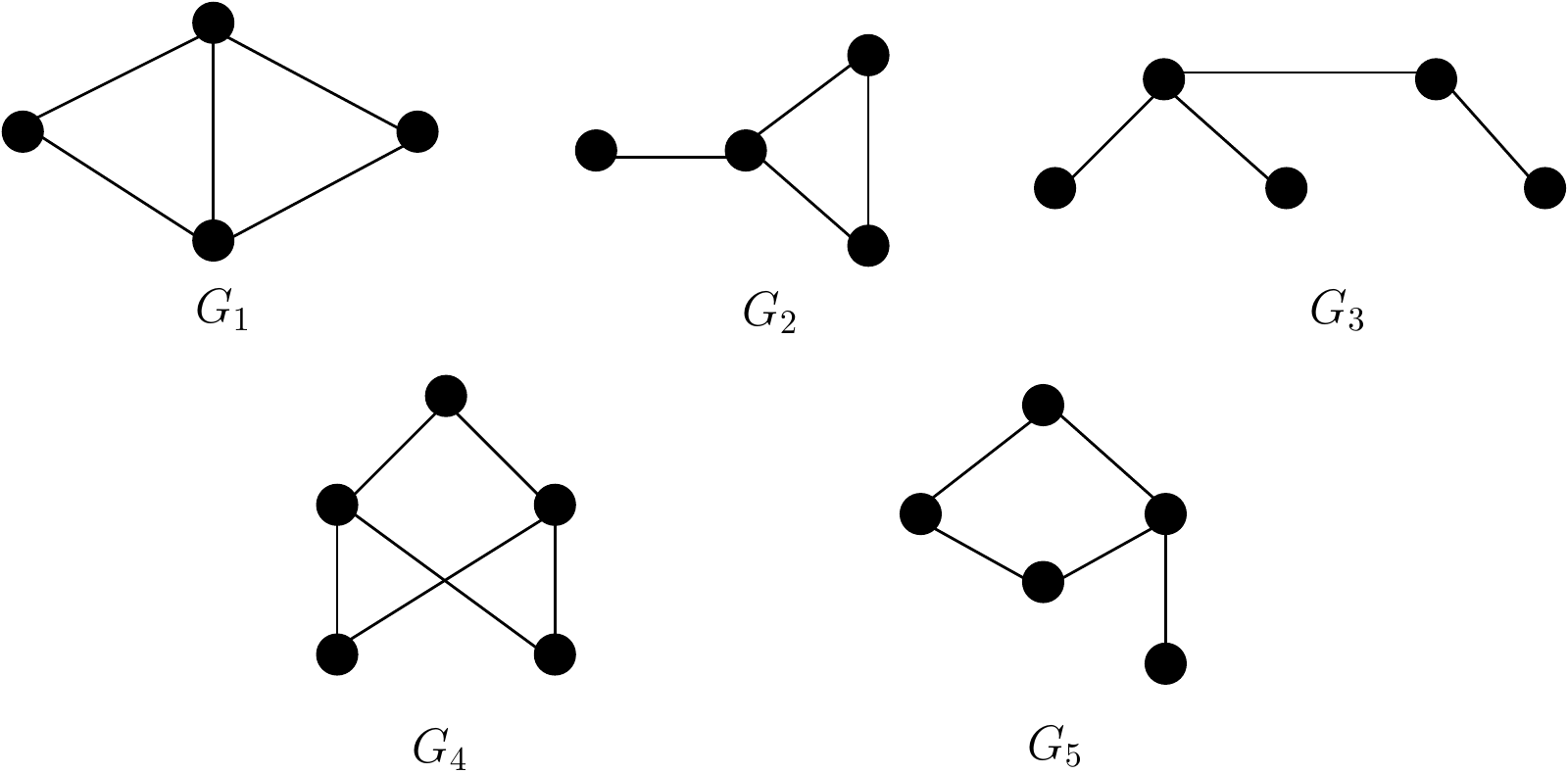}
\end{center}
\caption{$G_1$ is the diamond, $G_2$ is the paw, $G_3$ is the chair, and $G_4$ is $K_{2,3}$.}\label{fig:prohibited_graph}
\end{figure}

\begin{proof}
It is easy to check that $G_i$ does not satisfy the property $\mathcal P$ for each $1\le i\le 5$. 

Conversely, if $G$ does not satisfy the property $\mathcal P$, then $G$ contains, as an induced subgraph, one of the graphs depicted in Figure~\ref{fig:prohibited_graph}. Let $S$ be a subset of vertices of $G$ such that  $I_{P_3}[S]$ is properly contained in $I^{2}_{P_3}[S]$. Hence there exist two vertices $u$ and $v$ such that $u\in I_{P_3}[S]- S$ and $v\in I^{2}_{P_3}[S]- I_{P_3}[S]$. We will split the proof into two cases.
\vspace{0.5cm}

\textbf{Case 1: $N(v)\cap S\neq\emptyset$}. 
\vspace{0.5cm}

Assume first that there exists $z\in N(v)\cap N(u)\cap S$. Since $u\in I_{P_3}[S]$, there exists a vertex $w$ distinct of $z$ which is adjacent to $u$ and nonadjacent to $v$, because otherwise $v\in I_{P_3}[S]$.
If $w$ is adjacent to $z$, then $\{z,u,v,w\}$ either induces $G_1$ (if $v$ is adjacent to $u$) or  $\{z,u,v,w\}$ induces $G_2$ (if $v$ is nonadjacent to $u$). Hence $w$ is nonadjacent to $z$.  Besides, $u$ is nonadjacent to $v$, because otherwise $\{u,v,w,z\}$ induces $G_2$. Since $v\in I^2_{P_3}[S]$, there exists a vertex $x\in I^1_{P_3}[S]\setminus \{u\}$ which is adjacent to $v$. If $x$ is adjacent to $z$, then $\{u,v,x,z\}$ either induces $G_1$ (if $x$ is adjacent to $u$) or induces $G_2$ (if $x$ is nonadjacent to $u$). Hence $x$ is nonadjacent to $z$. On the one hand, if $x$ is adjacent to $u$ and $w$, then $\{u,v,w,x\}$ induces $G_1$. On the other hand, if $x$ is adjacent to $u$ and nonadjacent to $w$, then $\{u,v,w,x,z\}$ induces $G_5$. Thus $x$ is nonadjacent to $u$. Since $x\in I_{P_3}[S]-S$, there exists a vertex $s_1\in S\setminus\{w,z\}$ which is adjacent to $x$. If $x$ is adjacent to $w$ and $w$ is nonadjacent to $s_1$, then either $\{v,w,x,s_1,z\}$ induces $G_3$. In addition, if $x$ is adjacent to $w$ and $w$ is adjacent to $s_1$, then $\{v,s_1,w,x\}$ induces $G_2$. Hence $x$ is nonadjacent to $w$. Notice also that $s_1$ is nonadjacent to $z$. If $s_1$ is adjacent to $z$ and $u$, then $\{s_1,u,x,z\}$ induces $G_2$. If $s_1$ is adjacent to $z$ and nonadjacento to $u$, then $\{s_1,u,v,x,z\}$ induces $G_5$. Hence $s_1$ is nonadjacent to $z$.  Since $x$ is nonadjacent to $z$ and $w$, there exists a vertex $s\in S-\{s_1,w,z\}$ which is adjacent to $x$. By symmetry, $s_2$ is nonadjacent to $v$ and $z$. If $s_1$ is adjacent to $s_2$, then $\{s_1,s_2,v,x\}$ induces $G_2$. Therefore, if $s_1$ is nonadjacent to $s_2$, then $\{s_1,s_2,x,v,z\}$ induces $G_3$, a contradiction.
 
Assume now that $N(v)\cap N(u)\cap S=\emptyset$ and thus there exists a vertex $z\in S$ adjacent to $v$ but nonadjacent to $u$, and there exist two vertices $w_1,w_2\in S$ adjacent to $u$ and nonadjacent to $v$. Suppose, towards a contradiction, that $u$ is adjacent to $v$. Hence $w_1$ is nonadjacent to $w_2$, because otherwise $\{u,v,w_1,w_2\}$ induces $G_2$. On the one hand, if $z$ is adjacent to $w_1$ and $w_2$, then $\{u,v,w_1,w_2,z\}$ induces $G_4$. On the other hand, if $z$ is adjacent to exactly one of $w_1$ and $w_2$, then $\{u,v,w_1,w_2,z\}$ induces $G_5$. Hence $z$ is nonadjacent to $w_1$ and $w_2$. Thus $\{u,v,w_1,w_2,z\}$ induces $G_3$, a contradiction. 
Therefore, $u$ is nonadjacent to $v$ which implies that there exists a vertex $x\in I_{P_3}^1[S]-S$ adjacen to $v$. We may assume, by the discussion of the above paragraph, that $x$ is nonadjacent to $z$, otherwise $N(v)\cap N(x)\cap S\neq\emptyset$. Consequently, there exist two vertices $a$ and $b$ in $S-\{z\}$ adjacent to $x$. Since $G$ has no $G_2$ as induced subgraph, $a$ is nonadjacent to $b$. Besides, since $G$ has no $G_4$ neither $G_5$ as induced subgraph, $a$ and $b$ are nonadjacent to $z$. Therefore, $\{a,b,x,v,z\}$ induces $G_3$.
\vspace{0.5cm}
  
\textbf{Case 2: $N(v)\cap S=\emptyset$}. 
\vspace{0.5cm}

There exist two vertices $u_1,u_2\in I_{P_3}[S]\setminus S$ which are adjacent to $v$. Besides, there exists two vertices $x_1,x_2\in S$ which are adjacent to $u_1$. If $x_1$ is adjacent to $x_2$, then $\{u_1,v,x_1,x_2\}$ induces $G_2$. Hence $x_1$ is nonadjacent to $x_2$. On the one hand, if $x_1$ is adjacent to $u_2$ and $u_1$ is adjacent to $u_2$ then $\{u_1,v,u_1,u_2\}$ induces $G_1$. On the other hand, if $x_1$ is adjacent to $u_2$ and $u_1$ is nonadjacent to $u_2$ then $\{u_1,v,u_1,u_2\}$ induces $G_2$. Thus $x_1$ is nonadjacent to $u_2$. By symmetry, $x_2$ is nonadjacent to $u_2$. Consequently, $u_1$ is nonadjacent to $u_2$, because otherwise $\{u_1,u_2,v,x\}$ induces $G_2$ and thus $u_1$ is nonadjacent to $u_2$. Therefore, $\{u_1,u_2,v,x_1,x_2\}$ induces $G_3$.
\vspace{0.5cm}

We have already proved that in all possible cases the graph $G$ which does not satisfies the property $\mathcal P$ contains one of the graph depicted in Figure~\ref{fig:prohibited_graph} as induced subgraph. 
\end{proof}


\begin{coro}\label{cor: ultimo}
Let $G$ be a graph. If $G$ satisfies the property $\mathcal P$, then $g_{P_3}(G)= h_{P_3}(G)$.
\end{coro}

Corollary~\ref{cor: ultimo} shows that Theorem~\ref{thm: special class} is a characterization of a subclass of those graphs $G$ such that $h(H)=g(H)$ for every induced subgraph of $G$, characterized in~\cite{centeno}.

\section*{Acknowledgments}

This work was partially supported by ANPCyT PICT 2017-13152 and Universidad Nacional del Sur Grant PGI 24/L115.

\bibliographystyle{abbrv}
\bibliography{caterpillar}

\end{document}